\documentclass[english, 12pt]{amsart}

\usepackage{tikz}
\usepackage{aliascnt}
\usepackage{mathrsfs}
\usepackage[all,poly,knot]{xy}
\usepackage{hyperref}
\usepackage{comment}  
\usepackage{csquotes}   
\usepackage{amssymb,amsbsy,amsmath,amsfonts,amssymb,amscd,
	graphics,color,footmisc,fancyhdr,multicol,fancybox,
	graphicx,mathrsfs,rotating,ifthen,wasysym}

\usepackage{times}
\usepackage{pdfpages}
\usepackage{multirow}
\usepackage{nccrules}
\usepackage{textcomp}
\usepackage{comment}
\usepackage{framed}
\usepackage[all]{xy}
\usepackage{graphicx}
\usepackage{pgf,tikz}
\usepackage{mathrsfs}
\usetikzlibrary{arrows}
\usepackage{lipsum}
\usepackage{mathtools}

\DeclareMathOperator{\dif}{\text{\normalfont d}}

\DeclareMathOperator{\rank}{rank}

\DeclareMathOperator{\supp}{supp}
\DeclareMathOperator{\ord}{ord}

\newcommand{\BB}{\mathbb{B}}
\def\log{\mathrm{log}\,}

\theoremstyle{plain}

\newtheorem{thm}{Theorem}[section]

\newtheorem{pro}[thm]{Proposition}

\newtheorem{defi}[thm]{Definition}

\theoremstyle{remark}

\numberwithin{equation}{section}

\usepackage[top=1.5in, bottom=1.5in, left=1in, right=1in]{geometry}

 \usepackage[hyperpageref]{backref}
 
\usepackage{xcolor}
\hypersetup{
	colorlinks,
	linkcolor={red!50!black},
	citecolor={blue!62!black},
	urlcolor={blue!80!black}
}
\theoremstyle{plain}
\newcommand{\thistheoremname}{}
\newtheorem*{genericthm*}{\thistheoremname}
\newenvironment{namedthm*}[1]{\renewcommand{\thistheoremname}{#1}%
	\begin{genericthm*}}
	{\end{genericthm*}}

\newtheoremstyle{named}{}{}{\itshape}{}{\bfseries}{.}{.5em}{\thmnote{#3's }#1}
\theoremstyle{named}

\makeatletter
\newcommand\thankssymb[1]{\textsuperscript{\@fnsymbol{#1}}}
\makeatother
\begin{document} 
\title[Holomorphic mappings of maximal rank into projective spaces]{\bf Holomorphic mappings of maximal rank\\
into projective spaces}

\subjclass[2010]{32H30, 32A22}
\keywords{Nevanlinna theory, Second Main Theorem, holomorphic mapping, Wronskian, Fermat hypersurface}

\author{Dinh Tuan Huynh}

\address{Department of Mathematics, University of Education, Hue University, 34 Le Loi St., Hue City, Vietnam}
\email{dinhtuanhuynh@hueuni.edu.vn}

\begin{abstract} Let $1\leq p\leq n$ be two positive integers.
For a linearly nondegenerate holomorphic mapping $f\colon\mathbb{C}^p\rightarrow\mathbb{P}^n(\mathbb{C})$ of {\sl maximal rank} intersecting a family of hyperplanes in {\sl general position}, we obtain a Cartan's type Second Main Theorem in which the counting functions are truncated to level $n+1-p$. Our result strengthens the classical results of Stoll and Vitter, and interpolates the important works of Cartan and Carlson-Griffiths.
	\end{abstract}

\maketitle
\section{Introduction}
Nevanlinna theory was established about one hundred years ago~\cite{Nevanlinna1925} by comparing the growth of a holomorphic map $f:\mathbb{C}\rightarrow\mathbb{P}^1(\mathbb{C})$ with  the number of impacts of $f$ on discs with respect to a collection of $q\geq3$ distinct points in $\mathbb{P}^1(\mathbb{C})$. This quantifies the classical little Picard theorem, and extends the fundamental theorem of algebra from polynomials to meromorphic functions.
	
A short time after that, higher dimensional Nevanlinna theory has been developed by many authors. Cartan~\cite{Cartan1933} studied the frequency of intersection of a linearly nondegenerate entire curve
	$f\colon\mathbb{C}\rightarrow\mathbb{P}^n(\mathbb{C})$
	and a family of $q\geq n+2$ hyperplanes $\{H_i\}_{i\,=\,1,\dots,\, q}$ in {\sl general position}, and he obtained the following {\sl second main theorem}:
	\begin{equation}
	\label{Cartan SMT}
	(q-n-1)
	\,
	T_f(r)
	\,
	\leq
	\,
	\sum_{i=1}^q N_f^{[n]}(r,H_i)+S_f(r),
	\end{equation}
	where $T_f(r)$, $N_f^{[n]}(r,H_i)$ and $S_f(r)$ are standard notions in Nevanlinna theory, which will be described in section 2. The Wronskian technique introduced by Cartan is still an effect tool in value distribution theory until now~\cite{Ru2009,Ru2018,HX2022}.

	Meanwhile, independently, Weyl's~\cite{Weyl1938,Weyl1943}, Ahlfors \cite{Ahlfors1941} built value distribution theory for derived curves associated to entire holomorphic curves which also consists of Cartan's Second Main Theorem as a particular case. Stoll~\cite{Stoll53,Stoll54} extended the source space of the Weyl-Ahlfors theory to arbitrary dimensional $\mathbb{C}^m$ and more generally to parabolic manifolds (see \cite{HX2022} for recent development). Vitter \cite{Vitter77} obtained a Logarithmic Derivative Lemma for meromorphic mappings on $\mathbb{C}^m$ and  gave another proof of the estimate \eqref{Cartan SMT} for $f\colon\mathbb{C}^m\rightarrow\mathbb{P}^n(\mathbb{C})$.
	
In the early 1970s, Griffiths et al. \cite{Carlson-Griffiths 72, Griffiths-King73} refreshed the value distribution theory and established a  Second Main Theorem for the class of differentiably nondegenerate meromorphic mappings $f\colon \mathbb{C}^m\rightarrow V$, where $m\geq\dim V$. Notably, in this result, the truncation level of the counting functions  was decreased to the optimal $1$.
	
It  would be natural to seek a second main theorem for holomorphic mapping, reflexing the geometry of the source and the target spaces. In this paper, we settle this question for the class of holomorphic mappings of {\sl maximal rank}.
\begin{defi}
A holomorphic mapping $f\colon\mathbb{C}^p\rightarrow\mathbb{P}^n(\mathbb{C})$ is said to be of {\sl maximal rank} if at some point $z\in\mathbb{C}^p$, the differential $df$ satisfies $\rank(df)=\min\{p,n\}$.	
\end{defi}
When $p\geq n$, a holomorphic mapping $f\colon\mathbb{C}^p\rightarrow\mathbb{P}^n(\mathbb{C})$ of maximal rank is said to be {\sl differentiably nondegenerate}. Since a satisfactory Second Main Theorem for the class of these mappings was obtained \cite{Carlson-Griffiths 72}, it suffices to work on the case where $1\leq p\leq n$. Our result states as follows.
	
\begin{namedthm*}{Main Theorem} Let $1\leq p\leq n$ be two positive integers.
	Let $f:\mathbb{C}^p\rightarrow\mathbb{P}^n(\mathbb{C})$ be a linearly nondegenerate  holomorphic mapping of maximal rank, and let $\{H_i\}_{1\leq i\leq q}$ be a family of $q\geq n+2$ hyperplanes in general position in $\mathbb{P}^n(\mathbb{C})$. Then the following Second Main Theorem type estimate holds
		\begin{align}
		\label{smt statement}
		(
		q
		-
		n-1
		)\,
		T_{f}(r)\leq\sum_{i=1}^q N_{f}^{[n+1-p]}(r,H_i)
		+
		S_{f}(r),
		\end{align}
		where 
		$$S_f(r)=O(\log T_f(r)+\log r)\qquad\parallel
		$$ is a small error term compared with the order function.
		
	\end{namedthm*}
The truncation level for the counting functions in the above statement contains the geometric information of both source and target spaces of the holomorphic mappings. It strengthens the result of Stoll-Vitter \cite{Stoll54, Vitter77},  where the truncation level $n$ was obtained.  In the case where $p=1$, we  recover the classical Second Main Theorem of Cartan. When $p=n$, the above second main theorem  is a direct consequence of the Carlson-Griffiths' theory of equidimensional holomorphic mappings \cite{Carlson-Griffiths 72}. Thus our main theorem interpolates the above two important works.

As the first applications, by some standard arguments, a defect relation and a ramification theorem for holomorphic mappings of maximal rank are deduced directly from our Second Main Theorem. We then use them to study the degeneracy of holomorphic mappings into  a Fermat hypersurface and provide a Nevanlinna theoretic proof for a result of Etesse \cite[Theorem 2.2.2]{Etesse2023}. Furthermore, by very elementary arguments, our method also works well in the case where $f$ omits a Fermat hypersurface. These results strengthen the work of Green \cite{Green75}, by improving the lower degree bound. In the below statements, we put 
$$
\kappa(p,n)=
\max\{n+1-p,1\}
=
\begin{cases}
n+1-p,&\text{if}\quad 1\leq p< n\\
1,&\text{if} \quad n\leq p.
\end{cases}
$$

\begin{namedthm*}{Theorem A}[See Theorem \ref{degeneracy of holomorphic mapping into fermat hypersurface}]
	Let $p\geq 1,n\geq2$ be positive integers. Let $F$ be the Fermat hypersurface of degree $d$ in $\mathbb{P}^n(\mathbb{C})$, defined by the homogeneous polynomial 
	$$
	Q(\omega)=\sum_{i=0}^n\omega_i^d,
	$$
	where $\omega=[\omega_0:\omega_1:\dots:\omega_n]$ is a homogeneous coordinate of $\mathbb{P}^n(\mathbb{C})$.
	If $d>(n+1)\kappa(n-1,p)$, then the image of every holomorphic mapping $f\colon\mathbb{C}^p\rightarrow F$ of maximal rank lies in a hyperplane section.
\end{namedthm*}

\begin{namedthm*}{Theorem B}[See Theorem \ref{degeneracy of holomorphic mapping omitting fermat hypersurface}]
Let $p\geq 1,n\geq2$ be positive integers. Let $F$ be the Fermat hypersurface of degree $d$ in $\mathbb{P}^n(\mathbb{C})$, defined by the homogeneous polynomial 
$$
Q(\omega)=\sum_{i=0}^n\omega_i^d,
$$
where $\omega=[\omega_0:\omega_1:\dots:\omega_n]$ is a homogeneous coordinate of $\mathbb{P}^n(\mathbb{C})$.
If $d>(n+1)\kappa(p,n)$, then every holomorphic mapping $f\colon\mathbb{C}^p\rightarrow\mathbb{P}^n(\mathbb{C})$ of maximal rank omitting the Fermat hypersurface $F$ is algebraically degenerate.	
\end{namedthm*}

The proof of the Main Theorem requires only some the standard  techniques of Cartan~\cite{Cartan1933} and the representation of linearly dependent condition for family of holomorphic functions on $\mathbb{C}^p$ via the vanishing of the generalized Wronskians~\cite{Schmidt1980,Rot55,Fujimoto1985}. Especially, we employ the recent key observation due to Etesse \cite{Etesse2023}, saying that a subfamily of generalized Wronskians, called {\sl geometric generalized Wronskians}, is enough to guarantee the linearly dependence.

The paper is organized as follows. 
In Section~\ref{section: preparation}, we give a short introduction to higher dimensional Nevanlinna theory, and prepare some necessary materials about the generalized Wronskians. Details of the proof of the Main Theorem will be provided in Section~\ref{proof of the Main Theorem}. In the last section, we give some applications of the Main Theorem.

\section*{Acknowledgments}
We thank Song-Yan Xie for fruitful discussions and valuable comments that greatly improve the text.

\section{Preliminaries}
\label{section: preparation}
	
\subsection{Nevanlinna theory}
Let $z=(z_1,\dots,z_p)$ be the standard coordinate system of $\mathbb{C}^p$. Denote by $\|\cdot\|$ the Euclidean norm:
$$
\|z\|=\sqrt{\sum_{i=1}^{p}|z_i|^2}.
$$	
Set
$$
\alpha=\dif\!\dif^c\|z\|^2,\quad\beta=\dif\!\dif^c\log\|z\|^2,\quad\gamma=\dif^c\log \|z\|^2\wedge\beta^{p-1},
$$
where $\dif=\partial+\bar{\partial}$ and $\dif^c=\dfrac{i}{4\pi}(\bar{\partial}-\partial)$.

Let $\BB_r:=\{z\in\mathbb{C}^p:\|z\|< r\}\subset \mathbb{C}^p$ be the open ball in $\mathbb{C}^p$ of radius $r>0$ centered at the origin. Fix a truncation level $m\in \mathbb{N}\cup \{\infty\}$. For an effective divisor $E=\sum_i\alpha_i E_i$ on $\mathbb{C}^p$ where  $E_i$ are irreducible components and $\alpha_i\geq 0$,  the $m$-truncated degree of the divisor $E$ on  the balls is given by
	\[
	n^{[m]}(t,E)
	:=
\dfrac{1}{t^{2p-2}}\int_{\mathbb{B}_t\cap(\sum_i\min\{m,\alpha_i\}E_i)}\alpha^{p-1}
	\eqno
{{\scriptstyle (t\,>\,0)},}
	\]
the \textsl{truncated counting function at level} $m$ of $E$ is then defined by taking the logarithmic average
	\[
	N^{[m]}(r,E)
	\,
	:=
	\,
	\int_1^r \frac{n^{[m]}(t, E)}{t}\,\dif\! t
	\eqno
	{{\scriptstyle (r\,>\,1)}.}
	\]
	When $m=\infty$, for abbreviation we  write $n(t,E)$, $N(r,E)$ for $n^{[\infty]}(t,E)$, $N^{[\infty]}(r,E)$ respectively. 
	
Let $f\colon\mathbb{C}^p\rightarrow \mathbb{P}^n(\mathbb{C})$ be an entire holomorphic map having a reduced representation $f=[f_0:\cdots:f_n]$ in the homogeneous coordinates $[z_0:\cdots:z_n]$ of $\mathbb{P}^n(\mathbb{C})$. Let $D=\{Q=0\}$ be a divisor in $\mathbb{P}^n(\mathbb{C})$ defined by a homogeneous polynomial $Q\in\mathbb{C}[z_0,\dots,z_n]$ of degree $d\geq 1$. If $f(\mathbb{C}^p)\not\subset\supp D$, then $f^*D$ is a divisor on $\mathbb{C}^p$. We then define the \textsl{truncated counting function} of $f$ with respect to $D$ as
	\[
	N_f^{[m]}(r,D)
	\,
	:=
	\,
	N^{[m]}\big(r,f^*D\big),
	\]
	which measures the intersection frequency of $f(\mathbb{C}^p)$ with $D$. If $f^*D=\sum_i\mu_iE_i$, where $\mu_i>0
	$ and $\mu=\min_i\{\mu_i\}$, then we say that  $f$ is {\sl completely $\mu$--ramified} over $D$, with the convention that $\mu=\infty$ if $f(\mathbb{C}^p)\cap\supp D=\varnothing$. Next,
	the \textsl{proximity function} of $f$ associated to the divisor $D$ is given by
	\[
	m_f(r,D)
	\,
	:=
	\,
	\int_{\|z\|=r}
	\log
	\frac{\big\Vert f(z)\big\Vert_{\max}^d\,
		\Vert Q\Vert_{\max}}{\big|Q(f)(z)\big|}
	\,
	\gamma(z),
	\]
	where $\Vert Q\Vert_{\max}$ is the maximum  absolute value of the coefficients of $Q$ and where
	\begin{equation}
	\label{| |max definition}
	\big\Vert f(z)\big\Vert_{\max}
	:=
	\max
	\{|f_0(z)|,\dots,|f_n(z)|\}.
	\end{equation}
	Since $\big|Q(f)\big|\leq
	\left(\substack{d+n\\ n}
	\right)\,
	\Vert Q\Vert_{\max}\cdot\Vert f\Vert_{\max}^d$, we see that $m_f(r,D)\geq O(1)$ is bounded  from below by some constant.
	Lastly, the \textsl{Cartan order function} of $f$ is defined by
	\begin{align*}
	T_f(r)
	:=
	\int_{\|z\|=r}
	\log
	\big\Vert f(z)\big\Vert_{\max} \gamma(z).
	\end{align*}
	
	The Nevanlinna theory is then established by comparing the above three functions. It consists of two fundamental theorems (for  comprehensive expositions, see Noguchi-Winkelmann \cite{Noguchi-Winkelmann2014} and Ru \cite{Ru2021}).
	
	\begin{namedthm*}{First Main Theorem}\label{fmt} Let $f\colon\mathbb{C}^p\rightarrow \mathbb{P}^n(\mathbb{C})$ be a holomorphic map and let $D$ be a hypersurface of degree $d$ in $\mathbb{P}^n(\mathbb{C})$ such that $f(\mathbb{C}^p)\not\subset D$. Then one has the estimate
		\[
		m_f(r,D)
		+
		N_f(r,D)
		\,
		=
		\,
		d\,T_f(r)
		+
		O(1)
		\]
		for every $r>1$,
		whence
		\begin{equation}
		\label{-fmt-inequality}
		N_f(r,D)
		\,
		\leq
		\,
		d\,T_f(r)+O(1).
		\end{equation}
	\end{namedthm*}
	
	Hence the counting function is bounded from above by some multiple of the order function. The reverse direction is usually much harder, and one often needs to take the sum of the counting functions of many divisors. Such types of estimates are so-called {\sl second main theorems}.

	Throughout this paper, for an entire holomorphic map $f,$ the notation $S_f(r)$ means a real function of $r \in \mathbb{R}^+$ such that 
	\[
	S_f(r) \leq
	O(\log(T_f(r)))+ \epsilon\, \log r
	\]
	for every positive constant $\epsilon$ and every $r$ outside of a subset (depending on $\epsilon$) of finite Lebesgue measure of $\mathbb{R}^+$. In the case where $f$ is rational, we understand that $S_f(r)=O(1)$. In any case we always have
	
	\[
	\liminf_{r\rightarrow\infty}\dfrac{S_f(r)}{T_f(r)}
	=
	0.
	\]

	\subsection{Geometric generalized Wronskian}
Recall that for a collection of $n+1$ entire holomorphic functions $f_0,f_1,\dots,f_n$, their Wronskian is the determinant of the square matrix formed by $f_i$ and their derivatives up to order $n$, namely
$$
W(f_0,f_1,\dots,f_n)
:=
\det
\begin{pmatrix}
f_0&f_1&\dots&f_n\\
f_0'&f_1'&\dots&f_n'\\
\vdots&\vdots&\dots&\vdots\\
f_0^{(n)}&f_1^{(n)}&\dots&f_n^{(n)}
\end{pmatrix}.
$$
It is an important and well-known fact that the family $\{f_0,f_1,\dots,f_n\}$ is $\mathbb{C}$-linearly independent if and only if $W(f_0,f_1,\dots,f_n)\not\equiv 0$ on $\mathbb{C}$. Hence the linearly non-degeneracy of an entire holomorphic map into $\mathbb{P}^n(\mathbb{C})$ having reduced representation $[f_0:f_1:\dots :f_n]$ could be displayed as the non-vanishing everywhere of the Wronskian $W(f_0,f_1,\dots,f_n)$.

Passing to the higher dimensional case where a collection of $n+1$ holomorphic functions $f_i\colon\mathbb{C}^p\rightarrow \mathbb{C}$ ($0\leq i\leq n$) is considered. For a tuple of $p$ nonnegative integers $\alpha=(\alpha_1,\dots,\alpha_p)$, denote by $\Delta^{\alpha}$ the differential operator
\begin{equation}
\label{form of Delta alpha}
\Delta^{\alpha}
:=
\dfrac{\partial^{|\alpha|}}{\partial z_1^{\alpha_1}\dots\partial z_p^{\alpha_p}},
\end{equation}
having order $|\alpha|=\sum_{i=1}^{p}\alpha_i$. A straightforward computation \cite{Fujimoto1985} yields 
\begin{pro}
\label{vanishing order estimate and pole of logarithmic derivative}
Let $g$ be a nonzero holomorphic function on $\mathbb{C}^p$ such that $\dfrac{\Delta^{\alpha}g}{g}\not\equiv 0$ for a tuple of $p$ nonnegative integers $\alpha=(\alpha_1,\dots,\alpha_p)$. For a point $t_0\in\mathbb{C}^p$, denote by $\nu^{\infty}_{\alpha}(t_0)$ the pole order of $\dfrac{\Delta^{\alpha}g}{g}$ at $t_0$ and by $\nu^{0}(t_0)$ the zero order of $g$ at $t_0$. Then
$$
\nu^{\infty}_{\alpha}(t_0)
\leq
\min\{\nu^{0}(t_0),|\alpha|
\}.
$$
\end{pro}	

A collection $\mathcal{S}=\{\Delta^s\}_{0\leq s\leq n}$ of $n+1$ differential operators of the form \eqref{form of Delta alpha} is said to be {\sl admissible} if for each $0\leq s\leq n$, the order $|\Delta_s|:=i_{s_1}+\dots+i_{s_p}$ of
$$\Delta^s=
\dfrac{\partial^{i_{s_1}+\dots+i_{s_p}}}{\partial z_1^{i_{s_1}}\dots\partial z_p^{i_{s_p}}}$$   satisfies
$|\Delta_s|\leq s$. In particular when $s=0$, the operator $\Delta^0$ is the identity.

A generalized Wronskian of the family of holomorphic functions $\{f_0,\dots,f_n\}$ on $\mathbb{C}^p$ is the determinant of a square matrix of the form
$$
W_{\mathcal{S}}(f_0,f_1,\dots,f_n)
=
\det
\begin{pmatrix}
\Delta^0 (f_0)&\Delta^0{f_1}&\dots&\Delta^0(f_n)\\
\Delta^1(f_0)&\Delta^1(f_1)&\dots&\Delta^1(f_n)\\
\vdots&\vdots&\dots&\vdots\\
\Delta^n(f_0)&\Delta^n(f_1)&\dots&\Delta^n(f_n)
\end{pmatrix},
$$
where $\mathcal{S}=\{\Delta^s\}_{0\leq s\leq n}$ is an admissible family. When $p=1$, we recover the classical Wronskian. Similarly as in the one-dimensional case, generalized Wronskians enjoy the following fundamental property \cite{Rot55,Fujimoto1985,Schmidt1980,BD2010}.
\begin{thm}
\label{fundamental property of generalized wronskian about linearly independent}
Let $f_0,f_1,\dots f_n$ be holomorphic functions on $\mathbb{C}^p$. Then the family $\{f_i\}_{0\leq i\leq n}$ is linearly independent over $\mathbb{C}$ if and only if there exists a generalized Wronskian $W_{\mathcal{S}}(f_0,f_1,\dots,f_n)$ of the family $\{f_i\}_{0\leq i\leq n}$ such that $W_{\mathcal{S}}(f_0,f_1,\dots,f_n)\not\equiv 0$.
\end{thm}
Therefore, the generalized Wronskians can be employed to display the linearly nondegenerate condition for holomorphic mapping $f\colon\mathbb{C}^p\rightarrow\mathbb{P}^n(\mathbb{C})$. But in higher dimensional case, one has many admissible collections of differential operators and hence, one has several choices among generalized Wronskians. Etesse \cite[Theorem 1.4.1]{Etesse2023} observed that  the generalized Wronskian in the above  Theorem can be chosen in  a smaller subfamily, called {\sl geometric generalized Wronskians}. 

Each element in an admissible family $\mathcal{S}$ is determined by a word written in the lexicographic order with the alphabet $\{1,\dots,p\}$. Identifying the family $\mathcal{S}$ with the set of its corresponding words, we say that $\mathcal{S}$ is a {\sl full set} if and only if it satisfies the following property:

\begin{center}
	if a word $s$ belongs to $\mathcal{S}$, then so does every one of its subwords.
\end{center}

\begin{defi} A geometric generalized Wronskian of the family of $n+1$ holomorphic functions $\{f_0,\dots,f_n\}$ on $\mathbb{C}^p$ is a generalized Wronskian $W_{\mathcal{S}}(f_0,f_1,\dots,f_n)$ constructed from some full set $\mathcal{S}$.	
\end{defi}

Geometric generalized Wronskians inherit all properties of the classical ones, for instance, if $g$ is a holomorphic function on $\mathbb{C}^p$, then
$$
W_{\mathcal{S}}(gf_0,gf_1,\dots,gf_n)
=
g^{n+1}
W_{\mathcal{S}}(f_0,f_1,\dots,f_n),
$$
for any full set $\mathcal{S}$. With this new notion, Theorem~\ref{fundamental property of generalized wronskian about linearly independent} is strengthened as follows.
\begin{thm}(\cite[Theorem 1.4.1]{Etesse2023})
\label{etesse observation}
Let $f_0,f_1,\dots f_n$ be holomorphic functions on $\mathbb{C}^p$. Then the family $\{f_0,f_1,\dots,f_n\}$ is linearly independent over $\mathbb{C}$ if and only if there exists some geometric generalized Wronskian $$W_{\mathcal{S}}(f_0,f_1,\dots,f_n)$$  such that $W_{\mathcal{S}}(f_0,f_1,\dots,f_n)\not\equiv 0$.
\end{thm}

Applying this result to study holomorphic mappings from $\mathbb{C}^p$ with maximal rank, we obtain:
\begin{thm}
	\label{generalized wronskian for maximal rank holo mapping Cp to CPn}
Let $f\colon\mathbb{C}^p\rightarrow\mathbb{P}^n(\mathbb{C})$ be a linearly nondegenerate holomorphic mapping of maximal rank and let $[f_0:f_1:\dots:f_n]$ be a reduced representation of $f$. Then there exists an admissible family $\mathcal{S}=\{\Delta^s\}_{0\leq s\leq n}$ containing at least $p$ differential operators of order $1$ such that $W_{\mathcal{S}}(f_0,f_1,\dots,f_n)\not\equiv 0$. 
\end{thm}
\begin{proof}
	Let $X_0\in\mathbb{C}^p$ be a point where $df$ attains the maximal rank. After a translation, one may assume that $X_0=0$. Suppose that around $f(X_0)$, the map $f$ is given as $f=[1:f_1:\dots:f_n]$. Since this is a local problem, after a biholomorphism between  open neighborhoods of $0$, one can suppose that in some local coordinates $z_1,\dots,z_p$ of $\mathbb{C}^p$ around $0$, there is an index subset $\{i_1,\dots,i_p\}$ of $\{1,\dots,n\}$ such that $\{f_{i_1},\dots,f_{i_p}\}=\{z_1,\dots,z_p\}$. Using Theorem \ref{etesse observation}, there exists an admissible full set $\mathcal{S}=\{\Delta^s\}_{0\leq s\leq n}$ such that $W_{\mathcal{S}}(f_0,f_1,\dots,f_n)\not\equiv 0$. Suppose on the contrary that $\mathcal{S}$ contains at most $p-1$ differential operators of order $1$, then one can find some index $i_j$  such that the $i_j$th column of the matrix

$$ \begin{pmatrix}
	\Delta^0 (f_0)&\Delta^0({f_1})&\dots&\Delta^0(f_n)\\
	\Delta^1(f_0)&\Delta^1(f_1)&\dots&\Delta^1(f_n)\\
	\vdots&\vdots&\dots&\vdots\\
	\Delta^n(f_0)&\Delta^n(f_1)&\dots&\Delta^n(f_n)
\end{pmatrix}$$
is of the form
$$
 \begin{pmatrix}
z_{i_j}\\
0\\
\vdots\\
0
\end{pmatrix},
$$
and hence it is proportional to the first column. By the alternating property of the determinant, this yields $$
W_{\mathcal{S}}(f_0,f_1,\dots f_n)\equiv 0,
$$
a contradiction. Thus  $\mathcal{S}$ contains at least $p$ differential operators of order $1$, as required.
\end{proof}
\section{Proof of the Main Theorem}
	\label{proof of the Main Theorem}
	
\subsection{Notation and conventions}
Fix a reduced representation $[f_0:\dots:f_n]$ of $f$.
	Denote by $Q=\{1,\dots,q\}$ the index set. Assume that the  hyperplanes $H_1, \dots, H_q$ are defined by linear forms  $H_1^*, \dots, H_q^*$ given by
	$$
	H_i^*(\omega)=\sum_{j=0}^{n}a_{ij}\omega_j\qquad(\omega=(\omega_0,\omega_1,\dots,\omega_n)\in\mathbb{C}^{n+1}),
	$$
normalized so that	
$$
\|H^*_i\|
=
\sqrt{\sum_{j=0}^{n}|a_{ij}|^2}
=
1
\eqno
\scriptstyle{(1\,\leq\,i\,\leq\,q)}.
$$  

Since $f$ is linearly nondegenerate, by Theorem~\ref{generalized wronskian for maximal rank holo mapping Cp to CPn}, there exists an admissible family $\mathcal{S}=\{\Delta^s\}_{0\leq s\leq n}$ containing at least $p-1$ differential operators of order $1$ such that $W_{\mathcal{S}}(f_0,f_1,\dots,f_n)\not\equiv 0$. Hence each operator in this family has order at most $n+1-p$.

Put 
$$
g_i
=
\sum_{j=0}^{n}a_{i_j}f_j,\qquad (i\in Q),
$$
then for any subset of $n+1$ indexes $R\subset Q$, and for the above admissible set $\mathcal{S}=\{\Delta^s\}_{0\leq s\leq n}$, one has
\begin{equation}
\label{generalized wronskian of f and g compared}
W_{\mathcal{S}}(\{g_i\}_{i\in R})
=A_{R}
W_{\mathcal{S}}(f_0,f_1,\dots,f_n),
\end{equation}
where $$A_{R}=\det\big(\big(a_{ij}\big)_{i\in R,\, 0\leq j\leq n}\big)\not=0,$$
since $H_1,\dots, H_q$ are in general position. We also consider the logarithmic Wronskian $W^{\log}_{\mathcal{S}}(\{g_i\}_{i\in R})$ of the family $\{g_i\}_{i\in R}$ with respect to $\mathcal{S}$ defined as
$$
W^{\log}_{\mathcal{S}}(\{g_i\}_{i\in R})
=\det \begin{pmatrix}
1&1&\dots&1\\
\dfrac{\Delta^1(f_0)}{f_0}&\dfrac{\Delta^1(f_1)}{f_1}&\dots&\dfrac{\Delta^1(f_n)}{f_n}\\
\vdots&\vdots&\dots&\vdots\\
\dfrac{\Delta^n(f_0)}{f_0}&\dfrac{\Delta^n(f_1)}{f_1}&\dots&\dfrac{\Delta^n(f_n)}{f_n}
\end{pmatrix}.$$

\subsection{An a priori estimate}
	
Here is an implement of Cartan's Wronskian technique for holomorphic mapping.
	
\begin{pro}
There exists some constant $K>0$ depending only on the family $\{H_i\}_{i\in Q}$ such that
		\begin{align}
		\label{estimation of norm, statement}
		\|f(z)\|_{\max}^{q-n-1}
		\leq
		K
		\cdot
		\bigg(
		\dfrac{\prod_{i\in Q}| g_i(z)|}{|W_{\mathcal{S}}(f_0,f_1,\dots,f_n)(z)|}
		\bigg)
		\times
		\sum_{R\subset Q,\,|R|=n+1}
		\big|W^{\log}_{\mathcal{S}}(\{g_i\}_{i\in R})(z)\big|,\quad (z\in\mathbb{C}^p).
		\end{align}
	\end{pro}
	
\begin{proof} We follow the arguments in \cite[page~125, Lemma 4.2.3]{Noguchi-Winkelmann2014}.
Since  the family $\{H_i\}_{i\in Q}$ is in general position, for any point $Z\in \mathbb{P}^{n}(\mathbb{C})$,  there exists some index subset $S\subset Q$ with cardinality $|S|=q-n-1$ such that all the corresponding hyperplanes in the family $\{H_i\}_{i\in Q}$ miss $\omega$, namely
$\prod_{i\in S}\frac{|H_i^*(Z)|}{\|Z\|}\not=0		$. Consequently, by compactness argument, there exists some constant $C_1>0$ depending only on $\{H_i\}_{i\in Q}$ such that
\begin{align}
\label{estimate norm1}
\dfrac{1}{C_1}
<
\sum_{S\subset Q,\,|S|=q-n-1}
\prod_{i\in S}\bigg(\dfrac{|H_i^*(Z)|}{\|Z\|}\bigg)
<
C_1
\qquad\qquad\qquad
\scriptstyle{(\forall\,Z\,\in\,\mathbb{P}^{n}(\mathbb{C})).}
\end{align}
Putting $C(S)=Q\setminus S$, then each term in the middle of the above inequality can be rewritten as
\begin{equation}
\label{rewrite 1}
\prod_{i\in S}\bigg(\dfrac{|H_i^*(Z)|}{\|Z\|}\bigg)
=
\dfrac{\prod_{i\in Q}|H_i^*(Z)|}{\|Z\|^{q-n-1}}
\cdot
\prod_{i\in C(S)}\dfrac{1}{|H_i^*(Z)|}.
\end{equation}

Taking the sum on both sides of the above equality for all $S$ and using the lower bound of~\eqref{estimate norm1}, we receive
\[
\|Z\|^{q-n-1}
\leq
C_1
\cdot
\bigg(\prod_{i\in Q}|H_i^*(Z)|\bigg)
\cdot
\sum_{S\subset Q,\,|S|=q-n-1}
\dfrac{1}{\prod_{i\in C(S)}|H_i^*(Z)|}.
\]
Substituting $Z$ by $f(z)$ in the above inequality and noting that $\|f|_{\max}\leq \|f\|$, one obtains
\begin{align}
\|f(z)\|_{\max}^{q-n-1}
&\leq
C_1
\cdot
\bigg(\prod_{i\in Q}|A_i^*\circ f(z)|\bigg)
\cdot
\sum_{S\subset Q,\,|S|=q-n-1}
\dfrac{1}{\prod_{i\in C(S)}|H_i^*\circ f(z)|}\notag\\
&=
C_1
\cdot
\bigg(
\dfrac{\prod_{i\in Q}|g_i(z)|}{|W_{\mathcal{S}}(f_0,f_1,\dots,f_n)(z)|}
\bigg)\times\sum_{S\subset Q,\,|S|=q-n-1}
\dfrac{|W_{\mathcal{S}}(f_0,f_1,\dots,f_n)(z)|\,
}{\prod_{i\in C(S)}|g_i(z)|}\notag\\
&=
C_1
\cdot
\bigg(
\dfrac{\prod_{i\in Q}|g_i(z)|}{|W_{\mathcal{S}}(f_0,f_1,\dots,f_n)(z)|}
\bigg)\times
\sum_{R\subset Q,\,|R|=n+1}
\dfrac{|W_{\mathcal{S}}(\{g_i\}_{i\in R})(z)|\,
}{|A_R|\prod_{i\in R}|g_i(z)|}\notag\\
&\leq
K
\cdot
\bigg(
\dfrac{\prod_{i\in Q}|g_i(z)|}{|W_{\mathcal{S}}(f_0,f_1,\dots,f_n)(z)|}
\bigg)\times
\sum_{R\subset Q,\,|R|=n+1}
\dfrac{|W_{\mathcal{S}}(\{g_i\}_{i\in R})(z)|\,
}{\prod_{i\in R}|g_i(z)|},\notag
\end{align}
where $K=C_1.\max\{\dfrac{1}{|A_R|}: R\subset Q, |R|=n+1\}$, which yields the desired inequality.
\end{proof}

\subsection{A vanishing order estimate}
	
\begin{pro}
\label{estimation of zero divisor of w}
The following inequality  holds:
\begin{align}
\label{divisor inequality statement}
\sum_{i\in Q}(g_i)_0-(W_{\mathcal{S}}(f_0,f_1,\dots,f_n))_0
\leq
\sum_{i\in Q}
\sum_{z\in\mathbb{C}^p}\min\{\ord_z g_i,n+1-p\}\{z\}.
\end{align}
\end{pro}

\begin{proof}
First, since this is a pointwise inequality, therefore it's enough to prove that, for every fixed $z_0\in\mathbb{C}^p$, one has

\begin{align}
\label{estimate zero order of W, reduced, poniwise estimate}
\sum_{i\in Q}\ord_{z_0} g_i-
\ord_{z_0}W_{\mathcal{S}}(f_0,f_1,\dots,f_n)
\leq
\sum_{i\in Q}\min\{\ord_{z_0} g_i,n+1-p\}.
\end{align}	

Since the family $\{H_i\}_{1\leq i\leq q}$ is in general position, there exists an index subset $R=\{\lambda_0,\dots,\lambda_n\}\subset Q$ of cardinality $|R|=n+1$ such that $g_i(z_0)\not=0 $ for all $i\notin R$. Using \eqref{generalized wronskian of f and g compared}, the above inequality can be rewritten as
\begin{align}
\label{estimate zero order of W, reduced, poniwise estimate, simplify to set R of cardinality n+1}
\sum_{i=0}^n\ord_{z_0} g_{\lambda_i}-
\ord_{z_0}W_{\mathcal{S}}(g_{\lambda_0},g_{\lambda_1},\dots,g_{\lambda_n})
\leq
\sum_{i=0}^n\min\{\ord_{z_0} g_{\lambda_i},n+1-p\}.
\end{align}	
The left hand side of \eqref{estimate zero order of W, reduced, poniwise estimate, simplify to set R of cardinality n+1} is exactly the pole order at $z_0$ of the logarithmic Wronskian $$W^{\log}_{\mathcal{S}}(g_{\lambda_0},g_{\lambda_1},\dots,g_{\lambda_n}).$$
 Denote by $\mu_{\lambda_i}$ the zero order of $g_{\lambda_i}$ at $z_0$. Since the order of each operator $\Delta_s\in\mathcal{S}$ satisfies $|\Delta_s|\leq n+1-p$, it follows from Proposition~\ref{vanishing order estimate and pole of logarithmic derivative} that the pole order of $W^{\log}_{\mathcal{S}}(g_{\lambda_0},g_{\lambda_1},\dots,g_{\lambda_n})$ is bounded from above by $\sum_{i=0}^n\min\{\ord_{z_0} g_{\lambda_i},n+1-p\}$, whence concludes the proof.

\end{proof} 
	
\subsection{An application of the logarithmic derivative lemma}
Recalling the following version of the classical logarithmic derivative lemma in higher dimension due to Vitter \cite{Vitter77} (see also \cite{Fujimoto1985}).
\begin{namedthm*}{Logarithmic derivative Lemma}
Let $g$ be a nonconstant meromorphic function on $\mathbb{C}^p$. Then for any $p$-tuple of nonnegative integers $\alpha=(\alpha_1,\dots,\alpha_p)\not=(0,\dots,0)$, the following estimate holds
\[
m\bigg(r,\dfrac{D^{\alpha}g}{g}\bigg)
=O(\log^+ T_g(r)+\log r)\qquad\parallel.
\]	
\end{namedthm*}
Consequently, one has:	
\begin{pro}(\cite{Vitter77,Fujimoto1985})
\label{logarithmic derivative estimate for log wronskian g_lambda i}	
For any subset $R=\{\lambda_0,\lambda_1,\dots,\lambda_n\}\subset Q$ with $|R|=n+1$, one has the estimate
\begin{equation}
\label{estimate of logarithmic wronskian}
m\big(r,W^{\log}_{\mathcal{S}}(g_{\lambda_0},g_{\lambda_1},\dots,g_{\lambda_n})\big)
=O(\log^+T_f(r)+\log r)\qquad\parallel.
\end{equation}
\end{pro}
\begin{proof}	
For the sake of completeness, we include a proof here. Since $W^{\log}_{\mathcal{S}}(g_{\lambda_0},g_{\lambda_1},\dots,g_{\lambda_n})$ is a polynomial of $\dfrac{\Delta^s(g_{\lambda_i})}{g_{\lambda_i}}$, $1\leq s\leq n$, $0\leq i\leq n$, using the basic inequalities
\begin{equation*}
\label{basic inequality log+}
\log^+\bigg(\sum_{i=1}^mx_i\bigg)
\leq\sum_{i=1}^m\log^+x_i+\log m,\qquad\qquad
\log^+\bigg(\prod_{i=1}^mx_i\bigg)
\leq 
\sum_{i=1}^m\log^+x_i,
\end{equation*}
there exists some constant $M>0$ such that
\begin{equation}
\label{proximity of wronskian of g_i is dominated by Dmgi/g_i}
m\big(r,W^{\log}_{\mathcal{S}}(g_{\lambda_0},g_{\lambda_1},\dots,g_{\lambda_n})\big)
\leq
M+M\sum_{s=1}^n\sum_{i=0}^n m\bigg(r,\dfrac{\Delta^sg_{\lambda_i}}{g_{\lambda_i}}\bigg).
\end{equation}
Applying the Logarithmic Derivative Lemma, one receives
\begin{equation}
\label{proximity of Delta sg_i/g_i estimate via logarithmic derivative lemma}
m\bigg(r,\dfrac{\Delta^s g_{\lambda_i}}{g_{\lambda_i}}\bigg)
=O(\log^+ T_{g_{\lambda_i}}(r)+\log r)\qquad\parallel.
\end{equation}	
By the Cauchy-Schwarz inequality, one obtains
\[
\|g_{\lambda_i}(z)\|^2
\leq
\|H^*_{\lambda_i}|^2
\|f(z)\|^2
=|f(z)\|^2,\eqno \scriptstyle{(0\,\leq i\,\leq\, n,\,z\,\in\,\mathbb{C}^p)},
\]
which implies
\begin{equation}
\label{order function of glambda i vs order function of f}
T_{g_{\lambda_i}}(r)
\leq
T_f(r)
+O(1).
\end{equation}
Combining \eqref{proximity of wronskian of g_i is dominated by Dmgi/g_i}, \eqref{proximity of Delta sg_i/g_i estimate via logarithmic derivative lemma}, \eqref{order function of glambda i vs order function of f}, one gets the desired estimate.
\end{proof}	

\subsection{End of the proof of the Main Theorem}
	
First, by taking logarithm on both sides of~\eqref{estimation of norm, statement} and then taking the integrations on disc, one  receives
\begin{align}
	\label{T<=integer varphi+psi}
	\big(q-n-1)T_{f}(r)
	\leq
	\dfrac{1}{2\pi}
	\int_0^{2\pi}
	\log \varphi(re^{i\theta})
	\dif\theta
	+
	\dfrac{1}{2\pi}
	\int_0^{2\pi}
	\psi(re^{i\theta})\dif\theta
	+O(1),
\end{align}
	where
	\[
	\varphi
	=
	\dfrac{\prod_{i\in Q}| g_i(z)|}{|W_{\mathcal{S}}(f_0,f_1,\dots,f_n)(z)|},
	\ \ \ \ \ 
	\psi=
	\sum_{R\subset Q,\,|R|=n+1}
	\big|W^{\log}_{\mathcal{S}}(\{g_i\}_{i\in R})(z)\big|.
	\]
Using the Jensen formula, we estimate the fist term in the right hand side of \eqref{T<=integer varphi+psi} as follows:
\begin{align}
\dfrac{1}{2\pi}
\int_0^{2\pi}
\log|\varphi(re^{i\theta})|\dif\theta
\leq
N_{\varphi}(r,0)+O(1)\notag.
\end{align}
Using Proposition~\ref{estimation of zero divisor of w}, one receives
\begin{align*}
	(\varphi)_0
	&=
\sum_{i\in Q}(g_i)_0-(W_{\mathcal{S}}(f_0,f_1,\dots,f_n))_0\\
	&\leq
	\sum_{i\in Q}
	\sum_{z\in\mathbb{C}^p}\min\{\ord_z g_i,n+1-p\}\{z\},
	\end{align*}
	which yields
\begin{equation*}
N_{\varphi}(r,0)
\leq
\sum_{i\in Q}N_f^{[n+1-p]}(r,H_i).
\end{equation*}
Combining the above two estimates, one obtains
\begin{align}
\dfrac{1}{2\pi}
\int_0^{2\pi}
\log|\varphi(re^{i\theta})|\dif\theta
\leq
\sum_{i\in Q} N_{f}^{[n+1-p]}(r,H_i)
+
O(1).\notag
\end{align}

	Together this with~Proposition~\ref{logarithmic derivative estimate for log wronskian g_lambda i} and~\eqref{T<=integer varphi+psi}, the proof is finished.

	\section{Some applications}
	\label{section: applications}
\subsection{Holomorphic mappings from source spaces of arbitrary dimension}
Combining the Main Theorem with the classical result of Carlson-Griffiths, we have
\begin{thm} 
\label{arbitrary dimension}	
Let $p,n$ be two positive integers.
	Let $f:\mathbb{C}^p\rightarrow\mathbb{P}^n(\mathbb{C})$ be a linearly nondegenerate  holomorphic mapping of maximal rank, and let $\{H_i\}_{1\leq i\leq q}$ be a family of $q\geq n+2$ hyperplanes in general position in $\mathbb{P}^n(\mathbb{C})$. Then the following Second Main Theorem type estimate holds
	\begin{align}
	(
	q
	-
	n-1
	)\,
	T_{f}(r)\leq\sum_{i=1}^q N_{f}^{[\kappa(p,n)]}(r,H_i)
	+
	S_{f}(r),
	\end{align}
	where
	\begin{equation}
	\label{defintion of kappa}
	\kappa(p,n)=
	\begin{cases}
	n+1-p,&\text{if}\quad 1\leq p< n\\
	1,&\text{if} \quad n\leq p
	\end{cases}
	\end{equation}
and
	$$S_f(r)=O(\log T_f(r)+\log r)\qquad\parallel
	$$ is a small error term compared with the order function.
	
\end{thm}
\subsection{Defect relation and ramification theorem}
It is well-known that one can deduce from a Second Main Theorem type estimate  the following two standard inequalities, called the defect relation and ramification theorem. With the improvement in the truncation level in our Main Theorem, these results can be strengthened as well.
\begin{namedthm*}{Defect relation}
Let $f:\mathbb{C}^p\rightarrow\mathbb{P}^n(\mathbb{C})$ be a linearly nondegenerate  holomorphic mapping of maximal rank, and let $\{H_i\}_{1\leq i\leq q}$ be a family of $q\geq n+2$ hyperplanes in general position in $\mathbb{P}^n(\mathbb{C})$. Then one has the following defect relation
\begin{equation}
\label{defect relation}
\sum_{i=1}^q \delta_{f}^{[\kappa(p,n)]}(H_i)
\leq
n+1,
\end{equation}
where $\kappa(p,n)$ is given as in~\eqref{defintion of kappa}.
\end{namedthm*}
	
\begin{proof}
Rewriting Theorem \ref{arbitrary dimension}  as
\[
\sum_{i=1}^q
\bigg(
1-
\dfrac{N_f^{[\kappa(p,n)]}(r,H_i)}{T_{f}(r)}
\bigg)
\leq
n+1
+
\dfrac{S_{f}(r)}{T_{f}(r)},
\]
then taking the limit inferior of both sides of the above inequality, one obtains the claimed inequality.
\end{proof}
	
\begin{namedthm*}{Ramification theorem}
Let $f:\mathbb{C}^p\rightarrow\mathbb{P}^n(\mathbb{C})$ be a linearly nondegenerate holomorphic mapping of maximal rank, and let $\{H_i\}_{1\leq i\leq q}$ be a family of $q\geq n+2$ hyperplanes in general position in $\mathbb{P}^n(\mathbb{C})$. If $f$ is completely $\mu_{i}$--ramified over each  hyperplane $H_i$ for $i=1,\dots, q$, then one has
\[
\label{ramification theorem statement}
\sum_{i=1}^q
\bigg(
1
-
\dfrac{\kappa(p,n)}{\mu_{i}}
\bigg) 
\leq
n+1,
\]
where $\kappa(p,n)$ is given as in \eqref{defintion of kappa}.
\end{namedthm*}
\begin{proof}
For each $i$ with $\mu_i<\infty$, using the basis inequality
$$
N_f^{[\kappa(p,n)]}(r,H_i)
\leq
\kappa(p,n)
N_f^{[1]}(r,H_i),
$$
together with the First Main Theorem, one gets
\[
\delta_{f}^{[\kappa(p,n)]}(H_i)
\geq
1-\dfrac{\kappa(p,n)}{\mu_i}.
\]
When $\mu_i=\infty$, the image $f(\mathbb{C}^p)$ avoids $H_i$, hence we have the trivial equality $\delta_{f}^{[\kappa(p,n)]}(H_i)=1$. By these observations, the desired inequality follows immediately from the defection relation \eqref{defect relation}.
\end{proof}

\subsection{Holomorphic mapping into the Fermat hypersurface}
We employ the previous results to provide another proof for \cite[Theorem 2.2.2]{Etesse2023}.
\begin{thm}
\label{degeneracy of holomorphic mapping into fermat hypersurface}
Let $p\geq 1,n\geq2$ be positive integers. Let $F$ be the Fermat hypersurface of degree $d$ in $\mathbb{P}^n(\mathbb{C})$, defined by the homogeneous polynomial 
$$
Q(\omega)=\sum_{i=0}^n\omega_i^d,
$$
where $\omega=[\omega_0:\omega_1:\dots:\omega_n]$ is a homogeneous coordinate of $\mathbb{P}^n(\mathbb{C})$.
If $d>(n+1)\kappa(n-1,p)$, then the image of every holomorphic mapping $f\colon\mathbb{C}^p\rightarrow F$ of maximal rank lies in a hyperplane section.
\end{thm}
\begin{proof} We follow the arguments in \cite[Example 3.10.21]{Kob98}.
Let $H\cong\mathbb{P}^{n-1}(\mathbb{C})$ be the hyperplane in $\mathbb{P}^n(\mathbb{C})$ defined by $\sum_{i=0}^{n}\omega_i=0$. Suppose that $[f_0:f_1:\dots:f_n]$ is a reduced representation of $f$. Consider the endomorphism $$\pi\colon\mathbb{P}^n(\mathbb{C})\rightarrow\mathbb{P}^n(\mathbb{C}),\qquad [\omega_0:\omega_1:\dots:\omega_n]\rightarrow [\omega_0^d:\omega_1^d:\dots:\omega_n^d],
$$
and put $g=\pi\circ f\colon\mathbb{C}^p\rightarrow H\cong\mathbb{P}^{n-1}(\mathbb{C})$. Then $g$ is of maximal rank and its image lies in the hyperplane $H\cong\mathbb{P}^{n-1}(\mathbb{C})$. Let $\{H_i\}_{0\leq i\leq n}$ be the family of $n+1$ hyperplanes in $H\cong\mathbb{P}^{n-1}(\mathbb{C})$ given by $H_i=\{\omega_i=0\}$, which is in general position. For any $z\in g^{-1}(H_i)=f^{-1}(H_i)$, it is clear that $\ord_zg^*H_i\geq d$. If the image of $g$ doesn't lie in a smaller linear subspace of $H$, then by applying the ramification theorem for $g$ and the family $\{H_i\}_{0\leq i\leq n}$, one obtains
\[
\sum_{i=0}^{n}\bigg(1-\dfrac{\kappa(p,n-1)}{d}\bigg)\leq n,
\]
which yields $d\leq (n+1)\kappa(p,n-1)$, a contradiction. Hence the image of $g$ must be contained in some smaller linear subspace of $H$, which implies that the image of $f$ must lie in a hyperplane section.
\end{proof}

\subsection{Holomorphic mapping omitting a Fermat hypersurface}

\begin{thm}
\label{degeneracy of holomorphic mapping omitting fermat hypersurface}	
Let $p\geq 1,n\geq2$ be positive integers. Let $F$ be the Fermat hypersurface of degree $d$ in $\mathbb{P}^n(\mathbb{C})$, defined by the homogeneous polynomial 
	$$
	Q(\omega)=\sum_{i=0}^n\omega_i^d,
	$$
	where $\omega=[\omega_0:\omega_1:\dots:\omega_n]$ is a homogeneous coordinate of $\mathbb{P}^n(\mathbb{C})$.
	If $d>(n+1)\kappa(p,n)$, then  every holomorphic mapping $f\colon\mathbb{C}^p\rightarrow\mathbb{P}^n(\mathbb{C})$ of maximal rank omitting the Fermat hypersurface $F$ is algebraically degenerate.
\end{thm}
\begin{proof}
 Suppose that $[f_0:f_1:\dots:f_n]$ is a reduced representation of $f$. As in the previous part, we consider the endomorphism $$\pi\colon\mathbb{P}^n(\mathbb{C})\rightarrow\mathbb{P}^n(\mathbb{C}),\qquad [\omega_0:\omega_1:\dots:\omega_n]\rightarrow [\omega_0^d:\omega_1^d:\dots:\omega_n^d],
$$
and put $g=\pi\circ f\colon\mathbb{C}^p\rightarrow\mathbb{P}^n(\mathbb{C})$. Then $g$ is of maximal rank. Let $\{H_i\}_{0\leq i\leq n+1}$ be the family of $n+2$ hyperplanes in $\mathbb{P}^{n}(\mathbb{C})$ given by
\begin{align*}
H_i&=\{\omega_i=0\},\quad(0\leq i\leq n),\\
H_{n+1}&=\big\{\sum_{i=0}^{n}\omega_i=0\big\},
\end{align*}
 which is in general position. Since $f$ omits the Fermat hypersurface $F$, the holomorphic mapping $g$ avoids the hyperplane $H_{n+1}$. For each $i$ with $0\leq i\leq n$ and for any $z\in g^{-1}(H_i)=f^{-1}(H_i)$, it is clear that $\ord_zg^*H_i\geq d$. If the image of $g$ doesn't lie in any hyperplane, then by applying the ramification theorem for $g$ and the family $\{H_i\}_{0\leq i\leq n+1}$, one obtains
	\[
	\sum_{i=0}^{n}\bigg(1-\dfrac{\kappa(p,n)}{d}\bigg)+1\leq n+1,
	\]
	which yields $d\leq (n+1)\kappa(p,n)$, a contradiction. Hence the image of $g$ must be contained in some hyperplane, which implies that the image of $f$ must lie in some hypersurface.
\end{proof}

\begin{center}
\bibliographystyle{plain}

\begin{thebibliography}{10}
			
\bibitem{Ahlfors1941}
Lars~V. Ahlfors.
\newblock The theory of meromorphic curves.
\newblock {\em Acta Soc. Sci. Fennicae. Nova Ser. A.}, 3(4):31, 1941.
			
\bibitem{Cartan1933}
			Henri Cartan.
			\newblock Sur les z\'{e}ros des combinaisons lin\'{e}arires de $p$ fonctions
			holomorphes donn\'{e}es.
			\newblock {\em Mathematica}, 7:5--31, 1933.
			
			\bibitem{Fujimoto1985}
			Hirotaka Fujimoto.
			\newblock Non-integrated defect relation for meromorphic maps of complete K\"{a}hler manifolds into {$P^{N_1}({\bf C})\times\dots\times P^{N_k}({\bf C})$}.
			\newblock {\em Japan. J. Math.}, Vol. 11, No. 2, 1985.
			
	\bibitem{Ru2021}
	Min Ru.
	\newblock {\em Nevanlinna theory and its relation to Diophantine approximation}.
	\newblock  World Scientific Publishing, Second edition (2021).
			
	
	\bibitem{Nevanlinna1925}
			Rolf Nevanlinna.
			\newblock Zur {T}heorie der {M}eromorphen {F}unktionen.
			\newblock {\em Acta Math.}, 46(1-2):1--99, 1925.
			
				
			\bibitem{Noguchi-Winkelmann2014}
			Junjiro Noguchi and J\"{o}rg Winkelmann.
			\newblock {\em Nevanlinna theory in several complex variables and {D}iophantine
				approximation}, volume 350 of {\em Grundlehren der Mathematischen
				Wissenschaften [Fundamental Principles of Mathematical Sciences]}.
			\newblock Springer, Tokyo, 2014.
			
			\bibitem{Ru2009}
			Min Ru.
			\newblock Holomorphic curves into algebraic varieties.
			\newblock {\em Ann. of Math. (2)}, 169(1):255--267, 2009.
			
			\bibitem{Ru2018}
			Min Ru.
			\newblock A {C}artan's second main theorem approach in {N}evanlinna theory.
			\newblock {\em Acta Math. Sin. (Engl. Ser.)}, 34(8):1208--1224, 2018.
			
				
			\bibitem{Stoll53}
			Wilhelm Stoll.
			\newblock Die beiden {H}aupts\"{a}tze der {W}ertverteilungstheorie bei
			{F}unktionen mehrerer komplexer {V}er\"{a}nderlichen. {I}.
			\newblock {\em Acta Math.}, 90:1--115, 1953.
			
			\bibitem{Stoll54}
			Wolhelm Stoll.
			\newblock Die beiden {H}aupts\"{a}tze der {W}ertverteilungstheorie bei
			{F}unktionen mehrerer komplexer {V}er\"{a}nderlichen. {II}.
			\newblock {\em Acta Math.}, 92:55--169, 1954.
			
			\bibitem{Weyl1943}
			Hermann Weyl.
			\newblock { Meromorphic {F}unctions and {A}nalytic {C}urves}.
			\newblock {\em Annals of Mathematics Studies}, no. 12. Princeton University Press,
			Princeton, N. J., 1943.
			
			\bibitem{Weyl1938}
			Hermann Weyl and Joachim Weyl.
			\newblock Meromorphic curves.
			\newblock {\em Ann. of Math. (2)}, 39(3):516--538, 1938.
			
			\bibitem{Griffiths-King73}
	Phillip Griffiths, James King. \newblock Nevanlinna theory and holomorphic mappings between algebraic varieties.
	\newblock {\em Acta Math. 130} (1973), 145–220.
			
	\bibitem{Schmidt1980}
			Wolfgang M. Schmidt.
			\newblock {\em Diophantine approximation}.
			\newblock  Lecture Notes in Mathematics 785. Springer, Berlin (1980).
			
			\bibitem{Rot55} Klaus F. Roth.
			\newblock {Rational approximations to algebraic numbers}.
			\newblock {\em Mathematika}, 2:1–20, 1955
			
			\bibitem{Carlson-Griffiths 72} James Carlson and Phillip Griffiths.
			\newblock {A defect relation for equidimensional holomorphic mappings between algebraic varieties}.
			\newblock {\em Ann. Math. (2)}, 95: 557--584 (1972).
			
			\bibitem{Vitter77}	
			Al Vitter.
			\newblock {The lemma of the logarithmic derivative in several complex variables}.
			\newblock {\em Duke Math. J.}, 44, 89--104 (1977).
			
			\bibitem{Etesse2023}
			Antoine Etesse.
			\newblock {Geometric Generalized Wronskians: Applications in Intermediate Hyperbolicity and Foliation Theory}.
			\newblock{\em International Mathematics Research Notices}, to appear.
			
			\bibitem{BD2010}
			Alin Bostan and Philippe Dumas.
			\newblock {Wronskians and linear independence}.
			\newblock {\em American Mathematical Monthly}, 117, 8, 722--727, 2010.	
			
			\bibitem{HX2022}
			Dinh Tuan Huynh and Song-Yan Xie.
			\newblock {On the Weyl–Ahlfors theory of derived curves}.
			\newblock {\em Math. Z.},  300, 475--491, 2022.
			
\bibitem{Green75}
Mark Lee Green.
\newblock {Some Picard Theorems for Holomorphic Maps to Algebraic Varieties}.
\newblock {\em American Journal of Mathematics}, Vol. 97, No. 1 43--75 (Spring, 1975).
			
\bibitem{Kob98}
Shoshichi Kobayashi.
\newblock {\em Hyperbolic complex spaces}, volume 318 of {\em Grundlehren der Mathematischen Wissenschaften [Fundamental Principles of Mathematical Sciences]}.
\newblock { Springer-Verlag, Berlin, 1998.}
	
\end{thebibliography}

\end{center}
\address
\end{document}